\documentclass[12pt,reqno]{article}
\usepackage[usenames]{color}
\usepackage[colorlinks=true,linkcolor=webgreen,filecolor=webbrown,citecolor=webgreen]{hyperref}
\usepackage{amssymb}
\usepackage{amsmath}
\usepackage{amsfonts}
\usepackage[T1]{fontenc}
\usepackage{enumerate}
\usepackage{graphicx}
\usepackage{color}

\definecolor{webgreen}{rgb}{0,.5,0}
\definecolor{webbrown}{rgb}{.6,0,0}
\newtheorem{theorem}{Theorem}

\newtheorem{corollary}[theorem]{Corollary}

\newtheorem{lemma}[theorem]{Lemma}

\newenvironment{proof}[1][Proof]{\noindent\textbf{#1.} }{\ \rule{0.5em}{0.5em}}

\allowdisplaybreaks
\input{tcilatex}
\begin{document}

\begin{center}
\vskip1cm

{\LARGE \textbf{Simultaneous approximation by operators of exponential type}}

\vspace{2cm}

{\large Ulrich Abel}\\[0pt]
\textit{Technische Hochschule Mittelhessen}\\[0pt]
\textit{Fachbereich MND}\\[0pt]
\textit{Wilhelm-Leuschner-Stra\ss e 13, 61169 Friedberg, }\\[0pt]
\textit{Germany}\\[0pt]
\href{mailto:Ulrich.Abel@mnd.thm.de}{\texttt{Ulrich.Abel@mnd.thm.de}}
\end{center}

\vspace{2cm}

{\large \textbf{Abstract.}}

\bigskip

There are many results on the simultaneous approximation by sequences of
special positive linear operators. In the year 1978, Ismail and May as well
as Volkov independently studied operators of exponential type covering the
most classical approximation operators. In this paper we study asymptotic
properties of these class of operators. We prove that under certain
conditions, asymptotic expansions for sequences of operators belonging to a
slightly larger class of operators, can be differentiated term-by-term. This
general theorem contains several results which were previously obtained by
several authors for concrete operators. One corollary states, that the
complete asymptotic expansion for the Bernstein polynomials can be
differentiated term-by-term. This implies a well-known result on the
Voronovskaja formula obtained by Floater.

\bigskip



\vspace{2cm}


\section{Introduction}

\label{intro}


Let $e_{r}$ $\left( r=0,1,2,\ldots \right) $ denote the monomials defined by 
$e_{r}\left( x\right) =x^{r}$. For each real number $x$, define the function 
$\psi _{x}$ by $\psi _{x}=e_{1}-xe_{0}$.

Nearly all approximation operators $L_{n}$ possess asymptotic properties.
The most prominent instance are the Bernstein polynomials $B_{n}$. In 1932,
Voronovoskaja \cite{Voronovskaja-1932} showed the celebrated formula 
\begin{equation}
\lim_{n\rightarrow \infty }n\left( \left( B_{n}f\right) \left( x\right)
-f\left( x\right) \right) =\frac{1}{2}x\left( 1-x\right) f^{\prime \prime
}\left( x\right) ,  \label{voronovskaja}
\end{equation}%
for all real functions $f$ bounded on $\left[ 0,1\right] $ which have a
second order derivative. In that same year, Bernstein \cite{Bernstein-1932}
proved, for each $x\in \left( 0,1\right) $ and $q\in \mathbb{N}$, for which $%
f^{\left( 2q\right) }\left( x\right) $ exists, that 
\begin{equation}
\left( B_{n}f\right) \left( x\right) =\sum_{s=0}^{2q}\left( B_{n}\psi
_{x}^{s}\right) \left( x\right) \frac{f^{\left( s\right) }\left( x\right) }{%
s!}+o\left( n^{-q}\right) \text{ }\qquad \left( n\rightarrow \infty \right) .
\label{Bernstein-asymptotic}
\end{equation}%
Perhaps this is the first asymptotic expansion for any polynomial operator.
Noting that $\left( B_{n}\psi _{x}^{0}\right) \left( x\right) =1$, $\left(
B_{n}\psi _{x}^{1}\right) \left( x\right) =0$, $\left( B_{n}\psi
_{x}^{2}\right) \left( x\right) =x\left( 1-x\right) /n$, reveals that Eq.~$%
\left( \ref{voronovskaja}\right) $ is the special case $q=1$ of Eq.~$\left( %
\ref{Bernstein-asymptotic}\right) $. In other words, it is an asymptotic
expansion of order $1$ (as $\left( B_{n}f\right) \left( x\right) =f\left(
x\right) +x\left( 1-x\right) f^{\prime \prime }\left( x\right) /\left(
2n\right) +o\left( n^{-1}\right) $) and the constructive Weierstrass theorem
is one of order $0$ (as $\left( B_{n}f\right) \left( x\right) =f\left(
x\right) +o\left( 1\right) $) as $n\rightarrow \infty $. A more detailled
inspection of the central moments $B_{n}\psi _{x}^{s}$ shows that Eq.~$%
\left( \ref{Bernstein-asymptotic}\right) $ can be written as a complete
asymptotic expansion in the more explicit form 
\begin{equation}
\left( B_{n}f\right) \left( x\right) \sim \sum_{k=0}^{\infty }\frac{%
a_{k}\left( f,x\right) }{n^{k}}\text{ }\qquad \left( n\rightarrow \infty
\right) ,  \label{Bernstein-complete asymptotic expansion}
\end{equation}%
where the coefficients $a_{k}\left( f,x\right) $ are independent of $n$.
They are linear combinations of $f$ and finitely many derivatives of $f$
multiplied by certain polynomials. Explicit expressions for $a_{k}\left(
f,x\right) $ involving Stirling numbers of the first and second kind can be
derived \cite{abel-ATA2000a}. Eq.~$\left( \ref{Bernstein-asymptotic}\right) $
means that, for each $q=0,1,\ldots ,$ 
\begin{equation*}
\left( B_{n}f\right) \left( x\right) =\sum_{k=0}^{q}\frac{a\left( f,x\right) 
}{n^{k}}+o\left( n^{-q}\right) \text{ }\qquad \left( n\rightarrow \infty
\right) .
\end{equation*}%
Several authors studied asymptotic relations for the derivatives $\left(
B_{n}f\right) ^{\left( r\right) }$.

If a sequence of operators $L_{n}$ has a complete asymptotic expansion of
the form 
\begin{equation}
\left( L_{n}f\right) \left( x\right) \sim \sum_{k=0}^{\infty }\frac{%
a_{k}\left( f,x\right) }{n^{k}}\text{ }\qquad \left( n\rightarrow \infty
\right) ,  \label{Lnf-asymptotic}
\end{equation}%
it is quite natural to expect that the derivatives $\left( L_{n}f\right)
^{\left( r\right) }$ have a complete asymptotic expansion which can be
derived by term-by-term differentiation of the coefficients $a_{k}\left(
f,x\right) $ with respect to $x$, i.e., 
\begin{equation}
\left( L_{n}f\right) ^{\left( r\right) }\left( x\right) \sim
\sum_{k=0}^{\infty }\frac{a_{k}^{\left[ r\right] }\left( f,x\right) }{n^{k}}%
\text{ }\qquad \left( n\rightarrow \infty \right)
\label{Lnf-asymptotic-simultan}
\end{equation}%
with $a_{k}^{\left[ r\right] }\left( f,x\right) =\left( d/dx\right)
^{r}a_{k}\left( f,x\right) $. One of the first such formulas were shown for
the first derivative of the Bernstein--Kantorovich operators \cite[Theorem~3]%
{abel-ATA-1998}. The further step came in 2002 when A.-J. L\'{o}pez-Moreno,
J. Mart\'{\i}nez-Moreno, and F.-J. Mu\~{n}oz-Delgado \cite%
{Lopez-moreno-usw-Maratea-2002} established the asymptotic expansion for the
derivatives of the $B_{n}f$, namely 
\begin{equation*}
\left( B_{n}f\right) ^{\left( r\right) }\left( x\right) =f^{\left( r\right)
}\left( x\right) +\sum_{s=1}^{2q}\frac{1}{s!}\left( \left( B_{n}\psi
_{x}^{s}\right) \left( x\right) \cdot f^{\left( s\right) }\left( x\right)
\right) ^{\left( r\right) }+o\left( n^{-q}\right) \text{ }\qquad \left(
n\rightarrow \infty \right) ,
\end{equation*}%
provided that $f$ is sufficiently smooth on the interval $\left[ 0,1\right] $%
. The essential property was the so-called $\varphi $-convexity of the
operators. In \cite{Lopez-moreno-usw-JCAA-2003} A.-J. L\'{o}pez-Moreno and
F.-J. Mu\~{n}oz-Delgado extended the concept of $\varphi $-convexity in
order to prove general results that allow to calculate the asymptotic
formulae of the partial derivatives of sequences of multivariate linear
operators. In 2005, Floater \cite[Theorem 2]{Floater-JAT-2005} showed that
Voronovskaja's formula $\left( \ref{voronovskaja}\right) $ for the Bernstein
polynomials can be differentiated, i.e., 
\begin{equation}
\lim_{n\rightarrow \infty }n\left( \left( B_{n}f\right) ^{\left( r\right)
}\left( x\right) -f^{\left( r\right) }\left( x\right) \right) =\left( \frac{%
x\left( 1-x\right) f^{\prime \prime }\left( x\right) }{2}\right) ^{\left(
r\right) },  \label{Floater-Bernstein}
\end{equation}%
uniformly on $\left[ 0,1\right] $, for functions $f\in C^{r+2}\left[ 0,1%
\right] $. This result is a special case of \cite[Theorem 2]%
{Lopez-moreno-usw-JCAA-2003} which was found after completion of \cite%
{Floater-JAT-2005}. However, Floater used a completely different approach,
viz., remainder formulas for the Bernstein error in terms of divided
differences. This technique was refined by Abel \cite{Abel-RIMA-2019} for
simultaneous approximation by Bernstein--Chlodovsky polynomials. Further
results on asymptotic properties of sequences of operators can be found in 
\cite{Goodman-CA-1995}, \cite{Lopez-moreno-usw-JCAA-2003}, \cite%
{Walz-JCAM-2000}. Results on the differentiation of Voronovskaja's formula,
for general classes of operators including that of exponential type, can be
found in the recent work by \cite{Holhos-RIMA-2019}, \cite{Holhos-PMJ-2022}, 
\cite{Holhos-BullMalays-2021}. In the recent 2022 article Xiang studies
term-by-term differentiation of asymptotic expansions for certain Bernstein
type operators by probabilistic methods.

The purpose of this paper is to prove that operators of exponential type
possess, for sufficiently smooth functions $f$, a complete asymptotic
expansion $\left( \ref{Lnf-asymptotic}\right) $, which can be differentiated
term-by-term, i.e., $\left( \ref{Lnf-asymptotic-simultan}\right) $ holds
with $a_{k}^{\left[ r\right] }\left( f,x\right) =\left( d/dx\right)
^{r}a_{k}^{\left[ 0\right] }\left( f,x\right) $, for each positive integer $%
r $. The most prominent approximation operators are of exponential type.
Operators of exponential type preserve linear functions. In fact, we
consider a more general class of operators, which not necessarily have this
property, but preserve constant functions.

Using $a_{0}\left( f,x\right) =f\left( x\right) $, a direct consequence of
our results is the fact that the Voronovskaja-type formula, i.e., 
\begin{equation*}
\lim_{n\rightarrow \infty }n\left( \left( L_{n}f\right) \left( x\right)
-f\left( x\right) \right) =a_{1}\left( f,x\right) ,
\end{equation*}%
can by differentiated term-by-term, i.e.,%
\begin{equation*}
\lim_{n\rightarrow \infty }n\left( \left( L_{n}f\right) ^{\left( r\right)
}\left( x\right) -f^{\left( r\right) }\left( x\right) \right) =\left( \frac{d%
}{dx}\right) ^{r}a_{1}\left( f,x\right) .
\end{equation*}

\section{Operators of exponential type}

Let $I$ be a (finite or infinite)\ real interval. Recall that throughout the
paper we use the denotation $e_{r}\left( t\right) =t^{r}$ and $\psi
_{x}\left( t\right) =t-x$ $\left( t,x\in \mathbb{R}\right) $. In 1978,
Ismail and May \cite{Ismail-May-JMAA-1978} as well as Volkov \cite%
{Volkov-1978} studied linear operators $S_{n}:C\left( I\right) \rightarrow
C\left( I\right) $ which preserve constant functions and satisfy the
differential equation 
\begin{equation}
\left( S_{n}f\right) ^{\prime }\left( x\right) =\frac{n}{\varphi \left(
x\right) }\left( S_{n}\left( \psi _{x}f\right) \right) \left( x\right) ,
\label{ODE-exponential-type}
\end{equation}%
where $\varphi $ is an infinitely often differentiable function with $%
\varphi \left( x\right) \neq 0$ on $I$. If $I$ is an infinite interval, the
class $C\left( I\right) $ may be restricted to functions satisfying a
certain growth condition. Such operators $S_{n}$ are called operators of
exponential type. There is a huge literature dealing with them. Operators of
exponential type preserve linear functions, i.e., $S_{n}e_{r}=e_{r}$, for $%
r\in \left\{ 0,1\right\} $. The second moment $S_{n}e_{2}$, which is crucial
for the approximation properties of $S_{n}f$ as $n\rightarrow \infty $, is
given by $S_{n}e_{2}=e_{2}+\varphi /n$ (see \cite[Proposition 2.1]%
{Ismail-May-JMAA-1978}). Denote the central moments of the operators $S_{n}$
by 
\begin{equation*}
\mu _{n,s}\left( x\right) =\left( S_{n}\psi _{x}^{s}\right) \left( x\right) .
\end{equation*}%
In particular, we have $\mu _{n,0}=1$, $\mu _{n,1}=0$, and $\mu
_{n,2}=\varphi /n$. The recursive formula 
\begin{equation}
\mu _{n,s+1}=\frac{\varphi }{n}\left( s\mu _{n,s-1}+\mu _{n,s}^{\prime
}\right)  \label{recursive-exponential-type}
\end{equation}%
is an easy consequence of the differential equation\ $\left( \ref%
{ODE-exponential-type}\right) $ (see \cite[Eq. (3)]{Volkov-1978}, \cite[Eq.
(2.2)]{Ismail-May-JMAA-1978}). It easily delivers, for small values of $s$
the following instances:\ 
\begin{eqnarray*}
\mu _{n,3} &=&\varphi \varphi ^{\prime }n^{-2} \\
\mu _{n,4} &=&3\varphi ^{2}n^{-2}+\left( \varphi \left( \varphi ^{\prime
}\right) ^{2}+\varphi ^{2}\varphi ^{\prime \prime }\right) n^{-3} \\
\mu _{n,5} &=&10\varphi ^{2}\varphi ^{\prime }n^{-3}+\left( \varphi \left(
\varphi ^{\prime }\right) ^{3}+4\varphi ^{2}\varphi ^{\prime }\varphi
^{\prime \prime }+\varphi ^{3}\varphi ^{\prime \prime \prime }\right) n^{-4}
\\
\mu _{n,6} &=&15\varphi ^{3}n^{-3}+\left( \varphi ^{2}\left( \varphi
^{\prime }\right) ^{2}+15\varphi ^{3}\varphi ^{\prime \prime }\right) n^{-4}
\\
&&+\left( \varphi \left( \varphi ^{\prime }\right) ^{4}+11\varphi ^{2}\left(
\varphi ^{\prime }\right) ^{2}\varphi ^{\prime \prime }+4\varphi ^{3}\left(
\varphi ^{\prime \prime }\right) ^{2}+7\varphi ^{3}\varphi ^{\prime }\varphi
^{\prime \prime \prime }+\varphi ^{4}\varphi ^{\prime \prime \prime \prime
}\right) n^{-5}
\end{eqnarray*}

If $\varphi $ is a quadratic polynomial then $\mu _{n,s}\left( x\right) $ is
a polynomial in $x$ of degree at most $s$. For each $x\in I$, the recursive
formula implies that 
\begin{equation*}
\mu _{n,s}\left( x\right) =O\left( n^{-\left\lfloor \left( s+1\right)
/2\right\rfloor }\right) \text{ }\qquad \left( n\rightarrow \infty \right)
\qquad \text{for }s=0,1,\ldots .
\end{equation*}%
A further consequence of the recursive formula is the representation 
\begin{equation}
\mu _{n,s}\left( x\right) =\sum_{j=\left\lfloor \left( s+1\right)
/2\right\rfloor }^{s}g_{s,j}\left( x\right) n^{-j},
\label{representation-mu-by-g}
\end{equation}%
where $g_{s,j}$ are polynomials of degree at most $s$. In particular, $%
g_{0,0}\left( x\right) =1$, $g_{1,1}\left( x\right) =0$. Ismail and May \cite%
[Proposition 2 (iii)]{Ismail-May-JMAA-1978} showed that $g_{2s,2s}\left(
x\right) $ is a constant multiple of $\varphi ^{s}\left( x\right) $ and $%
g_{2s+1,2s+1}\left( x\right) $ is a constant multiple of $\varphi ^{s}\left(
x\right) \varphi ^{\prime }\left( x\right) $. In fact, it can be shown that 
\begin{eqnarray*}
g_{2s,2s}\left( x\right) &=&\frac{\left( 2s\right) !}{2^{s}s!}\varphi
^{s}\left( x\right) , \\
g_{2s+1,2s+1}\left( x\right) &=&\frac{s\cdot \left( 2s+1\right) !}{3\cdot
2^{s}s!}\varphi ^{s}\left( x\right) \varphi ^{\prime }\left( x\right) .
\end{eqnarray*}%
By Eq. $\left( \ref{representation-mu-by-g}\right) $, the relation 
\begin{equation*}
\left( S_{n}f\right) \left( x\right) =\sum_{s=0}^{2q}\left( \mu _{n,s}\left(
x\right) \frac{f^{\left( s\right) }\left( x\right) }{s!}\right) +o\left(
n^{-q}\right) \text{ }\qquad \left( n\rightarrow \infty \right) ,
\end{equation*}%
implies the complete asymptotic expansion 
\begin{equation*}
\left( S_{n}f\right) \left( x\right) =\sum_{k=0}^{q}\frac{1}{n^{k}}%
\sum_{s=k}^{2k}\frac{f^{\left( s\right) }\left( x\right) }{s!}g_{s,k}\left(
x\right) +o\left( n^{-q}\right) \text{ }\qquad \left( n\rightarrow \infty
\right) ,
\end{equation*}%
for appropriate functions sufficiently smooth at $x$.

\section{Main result}

Recently, Holho\c{s} \cite[Eq. (2)]{Holhos-RIMA-2019} (see also \cite[Eq.
(2.1)]{Holhos-PMJ-2022}) extended the class of exponential type operators,
by considering the more general differential equation 
\begin{equation}
\left( S_{n}f\right) ^{\prime }\left( x\right) =\frac{\lambda _{n}}{\varphi
\left( x\right) }\left( \left( S_{n}\left( \psi _{x}f\right) \right) \left(
x\right) -\left( S_{n}\psi _{x}\right) \left( x\right) \cdot \left(
S_{n}f\right) \left( x\right) \right) ,  \label{ODE-Holhos}
\end{equation}%
where $\left( \lambda _{n}\right) $ is a sequence of positive real numbers
tending to infinity. If $\varphi \left( x\right) $ is a quadratic polynomial
then $\mu _{n,s}\left( x\right) $ is a polynomial in $x$ of degree at most $s
$. Holho\c{s} \cite[Lemma 2]{Holhos-RIMA-2019} (cf. \cite[Lemma 2.3]%
{Holhos-PMJ-2022}) proved that, for each $x\in I$ and $s=0,1,\ldots $, the
central moments $\mu _{n,s}$ of operators satisfying the differential
equation $\left( \ref{ODE-Holhos}\right) $ have finite limits $%
\lim_{n\rightarrow \infty }\left( \lambda _{n}^{s}\cdot \mu _{n,2s}\left(
x\right) \right) $ and $\lim_{n\rightarrow \infty }\left( \lambda
_{n}^{s+1}\cdot \mu _{n,2s+1}\left( x\right) \right) $ and determined their
values. As a consequence, one has 
\begin{equation*}
\mu _{n,s}\left( x\right) =O\left( \lambda _{n}^{-\left\lfloor \left(
s+1\right) /2\right\rfloor }\right) \text{ }\qquad \left( n\rightarrow
\infty \right) .
\end{equation*}%
If the sequence $\left( \lambda _{n}\right) $ satisfies the inequalities $%
0<\alpha <\lambda _{n}/n<\beta $, for certain constants $\alpha ,\beta $, it
follows that 
\begin{equation}
\mu _{n,s}\left( x\right) =O\left( n^{-\left\lfloor \left( s+1\right)
/2\right\rfloor }\right) \text{ }\qquad \left( n\rightarrow \infty \right)
,\qquad \text{for }s=0,1,\ldots .  \label{condition-central-moments-holhos}
\end{equation}

Let $s$ be a positive integer. If $I$ is a finite interval, let $C_{s}\left(
I\right) $ be the space of all bounded functions on $I$. If $I$ is an
unbounded interval, let $C_{s}\left( I\right) $ be the space of all locally
bounded functions $f$ on $I$ which satisfy the growth condition $f\left(
t\right) =O\left( \left\vert t\right\vert ^{s}\right) $ as $\left\vert
t\right\vert \rightarrow +\infty $. For $x\in I$, let $C_{s,x}\left(
I\right) $ denote the space of all functions in $C_{s}\left( I\right) $
having a derivative of order $s$ at $x$. By Sikkema's theorem (Lemma~\ref%
{lemma-sikkema}), it follows that 
\begin{equation}
\left( S_{n}f\right) \left( x\right) =\sum_{s=0}^{2q}\left( \mu _{n,s}\left(
x\right) \frac{f^{\left( s\right) }\left( x\right) }{s!}\right) +o\left(
n^{-q}\right) \text{ }\qquad \left( n\rightarrow \infty \right) ,
\label{asymptotic-expansion}
\end{equation}%
for each function $f\in C_{2q,x}\left( I\right) $ and all positive integers $%
q$. Already Volkov \cite[THEOREM 2]{Volkov-1978} showed for operators of
exponential type the asymptotic relation 
\begin{equation*}
\left\vert \left( S_{n}f\right) \left( x\right) -\sum_{s=0}^{2q}\mu
_{n,s}\left( x\right) \frac{f^{\left( s\right) }\left( x\right) }{s!}%
\right\vert \leq K_{m}\left( x\right) n^{-q}\omega \left( f^{\left(
2q\right) },n^{-1/2}\right) \text{ }\qquad \left( n\in \mathbb{N}\right) ,
\end{equation*}%
where $K_{m}\left( x\right) $ is a constant independent of $n$, and $\omega $
denotes the ordinary modulus of continuity. We remark that in \cite%
{Costabile-ea-Maratea-1996-Proc-1998} these expansions were used to improve
the rate of convergence for operators of exponential type by applying
extrapolation techniques.

As our main result we show that under appropriate conditions the asymptotic
expansion $\left( \ref{asymptotic-expansion}\right) $ can be differentiated
term-by-term.

\begin{theorem}
\label{th-asymptotic-simultan}Let $r\in \mathbb{N}_{0}$, $q\in \mathbb{N}$
and $x\in I$. Suppose that the operators $S_{n}$ satisfy the differential
equation $\left( \ref{ODE-Holhos}\right) $ such that, for certain constants $%
\alpha ,\beta $, holds $0<\alpha <\lambda _{n}/n<\beta $. Then, for each
function $f\in C_{2\left( q+r\right) ,x}\left( I\right) $, the operators $%
S_{n}$ possess the (pointwise) asymptotic expansion 
\begin{equation}
\left( S_{n}f\right) ^{\left( r\right) }\left( x\right)
=\sum_{s=0}^{2q}\left( \mu _{n,s}\left( x\right) \frac{f^{\left( s\right)
}\left( x\right) }{s!}\right) ^{\left( r\right) }+o\left( n^{-q}\right) 
\text{ }\qquad \left( n\rightarrow \infty \right) ,  \label{expansion}
\end{equation}
\end{theorem}

In the special case $q=1$, we obtain the following Voronovskaja-type formula.

\begin{corollary}
\label{cor-Holhos-Voro-differentiated}Let $r\in \mathbb{N}_{0}$ and $x\in I$%
. For each function $f\in C_{2r+2,x}\left( I\right) $, the operators $S_{n}$
satisfying the differential equation $\left( \ref{ODE-Holhos}\right) $ with $%
0<\alpha <\lambda _{n}/n<\beta $, have the (pointwise) asymptotic relation 
\begin{equation*}
\left( S_{n}f\right) ^{\left( r\right) }\left( x\right) =\left( f\left(
x\right) +\mu _{n,1}\left( x\right) f^{\prime }\left( x\right) +\mu
_{n,2}\left( x\right) f^{\prime \prime }\left( x\right) \right) ^{\left(
r\right) }+o\left( n^{-q}\right) \text{ }\qquad \left( n\rightarrow \infty
\right) .
\end{equation*}
\end{corollary}

Because operators of exponential type preserve linear functions, they
trivially satisfy the differential equation $\left( \ref{ODE-Holhos}\right) $%
. Therefore, the conclusion of Theorem~\ref{th-asymptotic-simultan} is valid
also for operators of exponential type. Noting that they satisfy $\mu
_{n,1}\left( x\right) =0$ and $\mu _{n,2}=\varphi /n$, we obtain the
following corollary.

\begin{corollary}
\label{cor-exp-type-Voro-differentiated}Let $S_{n}$ be operators of
exponential type, $r\in \mathbb{N}_{0}$ and $x\in I$. For each function $%
f\in C_{2r+2,x}\left( I\right) $, it holds 
\begin{equation*}
\lim_{n\rightarrow \infty }n\left( \left( S_{n}f\right) ^{\left( r\right)
}\left( x\right) -f^{\left( r\right) }\left( x\right) \right) =\frac{1}{2}%
\left( \varphi \left( x\right) f\left( x\right) \right) ^{\left( r\right) }.
\end{equation*}
\end{corollary}

In the Bernstein case $\varphi \left( x\right) =x\left( 1-x\right) $,
Corollary~\ref{cor-exp-type-Voro-differentiated} recovers Floater's
celebrated formula $\left( \ref{Floater-Bernstein}\right) $.

\section{Auxiliary results and proofs}

Firstly, we recall the following general approximation theorem of Sikkema~%
\cite[Theorem 1]{Sikkema-1970}. It provides an asymptotic expansion for a
sequence of positive linear operators $L_{n}$ under a certain condition on
the central moments $\left( L_{n}\psi _{x}^{s}\right) \left( x\right) $.

\begin{lemma}
\label{lemma-sikkema} Let $I$ be a real interval, fix $x\in I$ and let $q\in 
\mathbb{N}$ be even. Furthermore, let $L_{n}:C\left( I\right) \rightarrow
C\left( I\right) $ be a sequence of positive linear operators which are
applicable to polynomials. Suppose that 
\begin{equation}
\left( L_{n}\psi _{x}^{s}\right) \left( x\right) =O\left( n^{-\left\lfloor
\left( s+1\right) /2\right\rfloor }\right) \text{ }\qquad \left(
n\rightarrow \infty \right) \qquad \text{for }s=0,1,\ldots ,q+2.
\label{sikkema1}
\end{equation}%
Then, for each function $f\in C_{q,x}\left( I\right) $,{\ possessing} a
derivative of order $q$ at $x$, it holds 
\begin{equation}
\left( L_{n}f\right) \left( x\right) =\sum_{s=0}^{q}\frac{f^{\left( s\right)
}\left( x\right) }{s!}\left( L_{n}\psi _{x}^{s}\right) \left( x\right)
+o\left( n^{-q}\right) \text{ }\qquad \left( n\rightarrow \infty \right) .
\label{sikkema2}
\end{equation}%
Furthermore, if $f^{\left( q+2\right) }\left( x\right) $ exists, the term $%
o\left( n^{-q}\right) $ in Eq.~$\left( \ref{sikkema2}\right) $ can be
replaced by $O\left( n^{-\left( q+1\right) }\right) $.
\end{lemma}

Holho\c{s} \cite[Page 3]{Holhos-PMJ-2022} proved the following extension of
the recursive formula $\left( \ref{recursive-exponential-type}\right) $.

\begin{lemma}
\label{lemma-recursive-formula-central-moments}For positive integers $s$,
the central moments $\mu _{n,s}$ satisfy the recursive formula 
\begin{equation*}
\frac{\lambda _{n}}{\varphi }\left( \mu _{n,s+1}-\mu _{n,1}\mu _{n,s}\right)
=s\mu _{n,s-1}+\mu _{n,s}^{\prime }.
\end{equation*}
\end{lemma}

\begin{lemma}
\label{lemma-derivative L psi^m f}For nonnegative integers $m$, it holds 
\begin{equation*}
\left( S_{n}\left( \psi _{x}^{m}f\right) \right) ^{\prime }\left( x\right) =%
\frac{\lambda _{n}}{\varphi \left( x\right) }\left( S_{n}\left( \psi
_{x}^{m+1}f\right) \left( x\right) -\left( S_{n}\psi _{x}\right) \left(
x\right) \left( S_{n}f\right) \left( x\right) \right) -m\left( S_{n}\left(
\psi _{x}^{m-1}f\right) \right) \left( x\right) .
\end{equation*}
\end{lemma}

\begin{proof}[Proof of Lemma~\protect\ref{lemma-derivative L psi^m f}]
The formula follows from the differential equation $\left( \ref{ODE-Holhos}%
\right) $.
\end{proof}

The next lemma directly follows by an application of the Leibniz rule for
derivatives.

\begin{lemma}
\label{lemma-derivative-psi-f}For nonnegative integers $m,s$, it holds $%
\left( \psi _{x}^{m}f\right) ^{\left( s\right) }\left( x\right) =\binom{s}{m}%
m!f^{\left( s-m\right) }\left( x\right) $, in particular 
\begin{equation}
\left( \psi _{x}f\right) ^{\left( s\right) }\left( x\right) =sf^{\left(
s-1\right) }\left( x\right) ,  \label{derivative-psi-f}
\end{equation}
\end{lemma}

Now we are in position to prove the main theorem.

\begin{proof}[Proof of Theorem~\protect\ref{th-asymptotic-simultan}]
We show that, for each integer $j\in \left\{ 0,1,\ldots ,r\right\} $, 
\begin{equation}
\left( S_{n}f\right) ^{\left( j\right) }\left( x\right) =\sum_{s=0}^{2\left(
q+r-j\right) }\left( \mu _{n,s}\left( x\right) \frac{f^{\left( s\right)
}\left( x\right) }{s!}\right) ^{\left( j\right) }+o\left( n^{-\left(
q+r-j\right) }\right) \text{ }\qquad \left( n\rightarrow \infty \right) .
\label{assertion-j}
\end{equation}%
By Sikkema's theorem (Lemma~\ref{lemma-sikkema}) or Volkov \cite[THEOREM 2]%
{Volkov-1978}, the formula is valid, for $j=0$. The proof uses mathematical
induction. Suppose that $\left( \ref{assertion-j}\right) $ is valid, for $%
0\leq j\leq m$, with a certain $m\in \left\{ 0,1,\ldots ,r-1\right\} $.
Then, by Lemma~\ref{lemma-recursive-formula-central-moments} and the Leibniz
rule for differentiation, 
\begin{align}
\left( S_{n}f\right) ^{\left( m+1\right) }\left( x\right) & =\left[ \frac{%
\lambda _{n}}{\varphi }\left( S_{n}\left( \psi _{x}f\right) -\mu
_{n,1}\left( S_{n}f\right) \right) \right] ^{\left( m\right) }\left( x\right)
\notag \\
& =\sum_{j=0}^{m}\binom{m}{j}\left( \frac{\lambda _{n}}{\varphi \left(
x\right) }\right) ^{\left( m-j\right) }\left( S_{n}\left( \psi _{x}f\right)
-\mu _{n,1}\left( S_{n}f\right) \right) ^{\left( j\right) }\left( x\right) .
\label{formula-Snf-derivative-m+1}
\end{align}%
For $0\leq j\leq m$, we conclude, by induction hypothesis and Lemma~\ref%
{lemma-derivative-psi-f}, 
\begin{eqnarray*}
&&\left( S_{n}\left( \psi _{x}f\right) \right) ^{\left( j\right) }\left(
x\right) \\
&=&\sum_{s=0}^{2\left( q+r-j\right) }\left( \mu _{n,s}\frac{\left( \psi
_{x}f\right) ^{\left( s\right) }}{s!}\right) ^{\left( j\right) }\left(
x\right) +o\left( n^{-\left( q+r-j\right) }\right) \\
&=&\sum_{s=0}^{2\left( q+r-j\right) -1}\left( \mu _{n,s+1}\frac{f^{\left(
s\right) }}{s!}\right) ^{\left( j\right) }\left( x\right) +o\left(
n^{-\left( q+r-j\right) }\right) \\
&=&\sum_{s=0}^{2\left( q+r-m\right) -1}\left( \mu _{n,s+1}\frac{f^{\left(
s\right) }}{s!}\right) ^{\left( j\right) }\left( x\right) +o\left(
n^{-\left( q+r-m\right) }\right) \text{ }\qquad \left( n\rightarrow \infty
\right) ,
\end{eqnarray*}%
since $\left( \mu _{n,s+1}\frac{f^{\left( s\right) }}{s!}\right) ^{\left(
j\right) }\left( x\right) =O\left( \mu _{n,s+1}\left( x\right) \right)
=O\left( n^{-\left\lfloor \left( s+2\right) /2\right\rfloor }\right)
=O\left( n^{-\left( q+r-m+1\right) }\right) $ as $n\rightarrow \infty $, if $%
s\geq 2\left( q+r-m\right) $. Furthermore, for $0\leq j\leq m$, we have 
\begin{eqnarray*}
&&\left( \mu _{n,1}\left( S_{n}f\right) \right) ^{\left( j\right) }\left(
x\right) \\
&=&\sum_{i=0}^{j}\binom{j}{i}\mu _{n,1}^{\left( j-i\right) }\left( x\right)
\left( S_{n}f\right) ^{\left( i\right) }\left( x\right) \\
&=&\sum_{i=0}^{j}\binom{j}{i}\mu _{n,1}^{\left( j-i\right) }\left( x\right)
\left( \sum_{s=0}^{2\left( q+r-i\right) }\left( \mu _{n,s}\frac{f^{\left(
s\right) }}{s!}\right) ^{\left( i\right) }\left( x\right) +o\left(
n^{-\left( q+r-i\right) }\right) \right) \\
&=&\sum_{s=0}^{2\left( q+r-j\right) }\sum_{i=0}^{j}\binom{j}{i}\mu
_{n,1}^{\left( j-i\right) }\left( x\right) \left( \mu _{n,s}\frac{f^{\left(
s\right) }}{s!}\right) ^{\left( i\right) }\left( x\right) +o\left(
n^{-\left( q+r-j\right) -1}\right) \\
&=&\sum_{s=0}^{2\left( q+r-j\right) }\left( \mu _{n,1}\mu _{n,s}\frac{%
f^{\left( s\right) }}{s!}\right) ^{\left( j\right) }\left( x\right) +o\left(
n^{-\left( q+r-j+1\right) }\right) \\
&=&\sum_{s=0}^{2\left( q+r-m\right) -1}\left( \mu _{n,1}\mu _{n,s}\frac{%
f^{\left( s\right) }}{s!}\right) ^{\left( j\right) }\left( x\right) +o\left(
n^{-\left( q+r-m\right) }\right) \text{ }\qquad \left( n\rightarrow \infty
\right) ,
\end{eqnarray*}%
since $\left( \mu _{n,1}\mu _{n,s}\frac{f^{\left( s\right) }}{s!}\right)
^{\left( i\right) }\left( x\right) =O\left( \mu _{n,1}\left( x\right) \mu
_{n,s}\left( x\right) \right) =O\left( n^{-1}n^{-\left\lfloor \left(
s+1\right) /2\right\rfloor }\right) =O\left( n^{-1}n^{-\left( q+r-m\right)
}\right) =O\left( n^{-\left( q+r-m+1\right) }\right) $ as $n\rightarrow
\infty $, if $s\geq 2\left( q+r-m\right) $. Inserting into Eq.~$\left( \ref%
{formula-Snf-derivative-m+1}\right) $ we obtain 
\begin{eqnarray*}
\left( S_{n}f\right) ^{\left( m+1\right) } &=&\sum_{j=0}^{m}\binom{m}{j}%
\left( \frac{\lambda _{n}}{\varphi \left( x\right) }\right) ^{\left(
m-j\right) }\left( \sum_{s=0}^{2\left( q+r-m\right) -1}\left( \mu
_{n,s+1}-\mu _{n,1}\mu _{n,s}\right) \frac{f^{\left( s\right) }}{s!}\right)
^{\left( j\right) }\left( x\right) \\
&&+\lambda _{n}\cdot o\left( n^{-\left( q+r-m\right) }\right)
\end{eqnarray*}%
as $n\rightarrow \infty $. By the Leibniz rule for differentiation, 
\begin{equation*}
\left( S_{n}f\right) ^{\left( m+1\right) }\left( x\right)
=\sum_{s=0}^{2\left( q+r-m\right) -1}\left( \frac{\lambda _{n}}{\varphi }%
\left( \mu _{n,s+1}-\mu _{n,1}\mu _{n,s}\right) \frac{f^{\left( s\right) }}{%
s!}\right) ^{\left( m\right) }\left( x\right) +\lambda _{n}\cdot o\left(
n^{-\left( q+r-m\right) }\right)
\end{equation*}%
as $n\rightarrow \infty $. By the recursive formula for the central moments
(Lemma~\ref{lemma-recursive-formula-central-moments}), 
\begin{eqnarray*}
\left( S_{n}f\right) ^{\left( m+1\right) }\left( x\right)
&=&\sum_{s=0}^{2\left( q+r-m\right) -1}\left( \left( s\mu _{n,s-1}+\mu
_{n,s}^{\prime }\right) \frac{f^{\left( s\right) }}{s!}\right) ^{\left(
m\right) }\left( x\right) +\lambda _{n}\cdot o\left( n^{-\left( q+r-m\right)
}\right) \\
&=&\sum_{s=0}^{2\left( q+r-m\right) -2}\left( \mu _{n,s}\frac{f^{\left(
s+1\right) }}{s!}+\mu _{n,s}^{\prime }\frac{f^{\left( s\right) }}{s!}\right)
^{\left( m\right) }\left( x\right) +\lambda _{n}\cdot o\left( n^{-\left(
q+r-m\right) }\right)
\end{eqnarray*}%
as $n\rightarrow \infty $, since $\left( \mu _{n,s}^{\prime }\frac{f^{\left(
s\right) }}{s!}\right) ^{\left( m\right) }\left( x\right) =O\left( \mu
_{n,s}\left( x\right) \right) =O\left( n^{-\left\lfloor \left( s+1\right)
/2\right\rfloor }\right) =O\left( n^{-\left( q+r-m\right) }\right) $, for $%
s=2\left( q+r-m\right) -1$. Since $\mu _{n,s}f^{\left( s+1\right) }+\mu
_{n,s}^{\prime }f^{\left( s\right) }=\left( \mu _{n,s}f^{\left( s\right)
}\right) ^{\prime }$, we obtain 
\begin{equation*}
\left( S_{n}f\right) ^{\left( m+1\right) }\left( x\right)
=\sum_{s=0}^{2\left( q+r-\left( m+1\right) \right) }\left( \mu _{n,s}\frac{%
f^{\left( s\right) }}{s!}\right) ^{\left( m+1\right) }\left( x\right)
+o\left( n^{-\left( q+r-\left( m+1\right) \right) }\right)
\end{equation*}%
as $n\rightarrow \infty $. Here we used the condition\ $0<\alpha <\lambda
_{n}/n<\beta $. This completes the proof.
\end{proof}

\strut

\thispagestyle{empty}

~\vfill

\end{document}